\documentclass[12pt,reqno,twoside]{amsart}
\usepackage{amsmath}
\usepackage{amssymb}

\usepackage{color}

\newcommand{\ignore}[1] {}

\newcommand{\vre}{\varepsilon}

\newtheorem{thm}{Theorem}

\newtheorem{prop}[thm]{Proposition}
\newtheorem{remark}[thm]{Remark}

\begin{document}

\title[Splitting fields]{Splitting fields of elements in arithmetic groups}

\author{Alexander Gorodnik}
\address{School of Mathematics \\ University of Bristol \\ Bristol, U.K.}
\email{a.gorodnik@bristol.ac.uk}
\thanks{The first author was supported in part by EPSRC, ERC, and RCUK}

\author{Amos Nevo}
\address{Department of Mathematics, Technion, Haifa, Israel}
\email{anevo@tx.technion.ac.il}
\thanks{The second author was supported by ISF grant}


\date{\today}



\begin{abstract}
We prove that the number of unimodular integral $n\times n$ matrices in a norm ball whose characteristic polynomial has Galois group different than the full symmetric group $S_n$ is of strictly lower order of magnitude than the number of all such matrices in the ball, as the radius increases. 
More generally, we prove a similar result for the Galois groups associated with elements in any connected semisimple linear algebraic group defined and simple over a number field $F$.  Our method is based on the abstract large sieve method developed by Kowalski, and the study of Galois groups via reductions modulo primes 
developed by Jouve, Kowalski and Zywina. The two key ingredients are a uniform quantitative lattice 
point counting result, and a non-concentration phenomenon for lattice points in algebraic subvarieties of the group variety, both established previously by the authors.
The results answer a question posed by Rivin and by Jouve, Kowalski and Zywina, who have considered Galois groups of random products of elements in algebraic groups. 
\end{abstract}

\maketitle

\section{Introduction}

Let $P(x)=x^d+a_{1}x^{d-1}+\cdots + a_{d-1}x+a_d$ be an irreducible polynomial with integral coefficients.
We denote by $\mathbb{Q}_P$ the splitting field of $P$. 
Since the Galois group $\hbox{Gal}(\mathbb{Q}_P/\mathbb{Q})$
acts on the roots of $P(x)$, it can be realised as a subgroup of the
symmetric group $S_d$. P.~Gallagher has shown in \cite{g} that typically
 the Galois group is, in fact, isomorphic to symmetric group
$S_d$. Namely,
$$
\left|\left\{P(x):\,\, 
\begin{tabular}{c}
$\max\{|a_{1}|,\ldots, |a_d|\}\le T$\\
$\hbox{Gal}(\mathbb{Q}_P/\mathbb{Q})\simeq S_d$
\end{tabular}
\right\}\right|
=(2T+1)^d+O_d\left(T^{d-1/2}\log T\right).
$$
The goal of this paper is to establish an analogous result for 
Galois groups of splitting fields of elements in arithmetic groups.
Let us consider, for instance, $\Gamma=\hbox{SL}_d(\mathbb{Z})$.
We denote by $\mathbb{Q}_\gamma$ the field generated by the eigenvalues of $\gamma$
(or, equivalently,
the splitting field of the characteristic polynomial $\det(x\cdot \hbox{Id}-\gamma)$).
Let $\|\cdot\|$ be a norm on $\hbox{Mat}_d(\mathbb{R})$, and 
$N_T(\Gamma):=|\{\gamma\in\Gamma:\, \|\gamma\|\le T \}|$. Then our main result below
implies that
$$
|\{\gamma\in\Gamma:\, \|\gamma\|\le T,\, \hbox{Gal}(\mathbb{Q}_\gamma/\mathbb{Q})\simeq S_d\}|=
N_T(\Gamma)+ O_{d,\vre}\left(N_T(\Gamma)^{1-\delta_d+\vre}\right)
$$
for all $\vre>0$,
where $\delta_2=1/80$, $\delta_d=d^{-3}(6d^2-4)^{-1}$ for even $d$, and
$\delta_d=d^{-2}(d-1)^{-1}(6d^2-4)^{-1}$ for odd $d$.

More generally, our standing assumptions will be that  ${\sf G}\subset \hbox{GL}_m$ is a connected semisimple algebraic group defined over a 
number field $F$, and that  $\sf G$ is simply connected and $F$-simple.
Let $S$ be a finite set of places of $F$ that contains all Archimedean places
such that $\sf G$ is isotropic over $S$.
We denote by $O_S$ the ring of $S$-integers in $F$, and consider the arithmetic group
$\Gamma:={\sf G}(O_S)$. For $\gamma\in \Gamma$, we denote by $F_\gamma$ the field
generated by the eigenvalues of $\gamma$.
We shall analyse the Galois groups $\hbox{Gal}(F_\gamma/F)$ for $\gamma\in\Gamma$
with $\gamma$'s indexed by the height function $\hbox{H}$, which is defined by
$$
\hbox{H}(\gamma):=\prod_{v\in S} \hbox{H}_v(\gamma),
$$
where the local heights $\hbox{H}_v$ are 
$$
\hbox{H}_v(\gamma):=\left\{
\begin{tabular}{ll}
$\left(\sum_{i,j} |\gamma_{ij}|_v^2\right)^{1/2}$ & for Archimedean places $v\in S$,\\
$\max_{ij} |\gamma_{ij}|_v$ & for non-Archimedean places $v\in S$.
\end{tabular}
\right.
$$
We set
$$
N_T(\Gamma):=|\{\gamma\in\Gamma:\, \hbox{H}(\gamma)\le T \}|.
$$

Let $F_{\sf G}:=\cap_{\sf T} F_{\sf T}$ where the intersection is taken over
all maximal tori $\sf T$ of $\sf G$ defined over $F$, and $F_{\sf T}$ denotes the splitting
field of the torus $\sf T$. We also denote by $W({\sf G})\simeq N_{\sf G}({\sf T})/C_{\sf G}({\sf T})$
the Weyl group of $\sf G$. Our first result shows that typically the Galois groups
$\hbox{Gal}(F_\gamma/(F_\gamma\cap F_{\sf G}))$ are isomorphic to the Weyl groups $W({\sf G})$.

\begin{thm}\label{th:split}
There exists $\delta>0$ such that
$$
\left|\left\{\gamma\in\Gamma:\,\, 
\begin{tabular}{l}
$\hbox{\rm H}(\gamma)\le T$, $F_\gamma\supset F_{\sf G}$\\
$\hbox{\rm Gal}(F_\gamma/F_{\sf G})\simeq W({\sf G})$
\end{tabular}
\right\}\right|
=N_T(\Gamma)+ O\left(N_T(\Gamma)^{1-\delta}\right).
$$
\end{thm}

In general, the Galois groups $\hbox{\rm Gal}(F_\gamma/F)$ are typically
isomorphic to a larger group $\Pi({\sf G})$ which we now define.
Let $\sf T$ be a maximal torus of $\sf G$ defined over $F$, and let $X({\sf T})$ be
the character group of $\sf T$ which is a free abelian group of rank $\dim({\sf T})$.
We denote by 
$\Pi({\sf G})$ the subgroup of $\hbox{Aut}(X({\sf T}))$ generated by 
the action of the Weyl group $W({\sf G})$ and the action of the Galois
group $\hbox{Gal}(F_{\sf T}/F)$. We note that the definition of the group $\Pi({\sf G})$
does not depend on a choice of the torus $\sf T$.

\begin{thm}\label{th:general}
There exists $\delta>0$ such that
$$
\left|\left\{\gamma\in\Gamma:\,\, 
\begin{tabular}{c}
$\hbox{\rm H}(\gamma)\le T$\\
$\hbox{\rm Gal}(F_\gamma/F)\simeq \Pi({\sf G})$
\end{tabular}
\right\}\right|
=N_T(\Gamma)+ O\left(N_T(\Gamma)^{1-\delta}\right).
$$
\end{thm}

\begin{remark} {\rm 
The exponent $\delta$ can given explicitly, and Theorems \ref{th:split} and \ref{th:general} 
hold with
\begin{equation}\label{eq:delta}
\delta<a(a+[F:\mathbb{Q}]\dim({\sf G}))^{-1}(2n_e(p_S))^{-1}(3\dim({\sf G})+1)^{-1},
\end{equation}
where $a$ is the H\"older exponent of the height balls,
$p_S$ is the integrability exponent of the relevant automorphic representations,
and $n_e(p)$ is the least even integer $\ge p/2$ if $p>2$ and $1$ if $p=2$.
We refer to \cite[Sec.~4]{gn} for this notation.
We note that in many cases we have $a=1$ (see 
\cite[Rem.~4.2]{gn} and \cite[Th.~3.15]{gn_book});
for instance, this is so when $S$ contains only Archimedean places.
Also, when the local height functions $\hbox{H}_v$, $v\in S$, are bi-invariant under a good special subgroup
of ${\sf G}(F_v)$, one can replace $2n_e(p_S)$ by $p_S$ (see \cite[Rem.~4.2]{gn}). 
}\end{remark}

We note that I. Rivin \cite{r} has raised a number of important questions on genericity properties in arithmetic lattices and mapping class groups. The present paper is motivated also by the works  
of F.~Jouve, E.~Kowalski, and D.~Zywina \cite{j,jkz,k}
who studied Galois groups of elements generated by random walks
and have established definitive results in this setting.
We are not aware of previous results about Galois groups
of elements indexed by the height function, a  question that was raised explicitly in \cite{r} and \cite[\S 7]{jkz}.

The proofs of the theorems utilize the abstract
large sieve method developed in Kowalski's book \cite{k}, 
and rely also on the technique of studying Galois groups via
reductions modulo primes that has been developed in great generality in \cite{jkz}.
Our arguments are based on the general counting results for congruence subgroups proved in \cite{gn},  which provide  the crucial spectral estimate 
necessary for the large sieve method to proceed (see equation (\ref{spectral estimate}) below). In addition, the non-concentration phenomenon 
established for subvarieties of semisimple group varieties in \cite{gn2}  is used to immediately reduce the computation of splitting 
fields to regular semisimple elements only, as non-regular elements  have a-priori lower rate of growth.

\section{The large sieve for arithmetic groups}

For a prime ideal $\mathfrak{p}$ of the ring of integers of $F$, 
we denote by ${\sf G}^{(\mathfrak{p})}$
the reduction of $\sf G$ modulo $\mathfrak{p}$. 
For almost all $\mathfrak{p}$, ${\sf G}^{(\mathfrak{p})}$ is a smooth connected 
algebraic group defined over the residue field ${\bf F}_\mathfrak{p}$.
We set $Y_{\mathfrak{p}}:={\sf G}^{(\mathfrak{p})}({\bf F}_\mathfrak{p})$,
and more generally, for a square-free ideal $\mathfrak{a}$, we set
$Y_{\mathfrak{a}}:=\prod_{\mathfrak{p}|\mathfrak{a}}Y_{\mathfrak{p}}$.
When the ideal $\mathfrak{a}$ is coprime to $S$,
we have a well-defined reduction map 
$$
\pi_{\mathfrak{p}}:\Gamma={\sf G}(O_S)\to Y_{\mathfrak{p}}.
$$
Given a family of subsets $\Omega_{\mathfrak{p}}\subset Y_{\mathfrak{p}}$ with $\mathfrak{p}\in\mathcal{L}^*$,
we define the sifted set by
$$
S_T\left(\Gamma,\{\Omega_{\mathfrak{p}}\}_{\mathfrak{p}\in\mathcal{L}^*};\mathcal{L}^*\right)
:=|\{\gamma\in\Gamma:\,\, \hbox{\rm H}(\gamma)\le T,\, 
\pi_{\mathfrak{p}}(\gamma)\notin \Omega_{\mathfrak{p}}\hbox{ for $\mathfrak{p}\in\mathcal{L}^*$}\}|.
$$
A fundamental problem in the sieve theory is to produce an upper estimate on the cardinality of this set.

\begin{prop}\label{p:large}
There exist a finite set $R$ of prime ideals containing $S$ and
constants $C,T_0,\rho>0$, depending only on $\Gamma$,
such that for any choice of 
\begin{itemize}
\item a set $\mathcal{L}^*$ of prime ideals coprime to $R$,
\item a set $\mathcal{L}$ of square-free ideals divisible by only prime ideals in $\mathcal{L}^*$,
\item a family of subsets $\Omega_{\mathfrak{p}}\subset Y_{\mathfrak{p}}$ with
  $\mathfrak{p}\in\mathcal{L}^*$,
\end{itemize}
the following estimate holds
$$
\left|S_T\left(\Gamma,\{\Omega_{\mathfrak{p}}\}_{\mathfrak{p}\in\mathcal{L}^*};\mathcal{L}^*\right)\right|
\le \frac{N_T(\Gamma)+C\, N_T(\Gamma)^{1-\rho}M(\mathcal{L})}{V\left(\{\Omega_{\mathfrak{p}}\}_{\mathfrak{p}\in\mathcal{L}^*}\right)},
$$
for all $T\ge T_0$, where
\begin{align}
M(\mathcal{L})&:=\max_{\mathfrak{a}\in\mathcal{L}} 
\sum_{\mathfrak{b}\in\mathcal{L}}|Y_{[\mathfrak{a},\mathfrak{b}]}|
\cdot |Y_{\mathfrak{a}}|^{1/2}\cdot |Y_{\mathfrak{b}}|^{1/2},\label{eq:m}\\
V\left(\{\Omega_{\mathfrak{p}}\}_{\mathfrak{p}\in\mathcal{L}^*}\right)&:={\sum_{\mathfrak{a}\in\mathcal{L}}}
\prod_{\mathfrak{p}|\mathfrak{a}}\frac{|\Omega_{\mathfrak{p}}|}{|Y_{\mathfrak{p}}|-|\Omega_{\mathfrak{p}}|}. \label{eq:v}
\end{align}
\end{prop}

\begin{proof}
We use the general version of the large sieve developed in \cite[Ch.~2]{k}.
We equip the spaces $Y_{\mathfrak{a}}$ with the uniform probability measure
and choose an orthonormal basis of $\mathcal{B}_{\mathfrak{a}}$ of $L^2(Y_{\mathfrak{a}})$
that contains the constant function $1$. 
We follow the convention
of \cite{k} and construct
the basis elements of $L^2(Y_{\mathfrak{a}})$ as products of basis elements of $L^2(Y_{\mathfrak{p}})$
with prime ideals $\mathfrak{p}$ dividing $\mathfrak{a}$.

It will be convenient to introduce a measure  $\mu_T=\sum_{\gamma\in\Gamma:\, \hbox{\tiny H}(\gamma)\le T} \delta_\gamma$
on $\Gamma$, where $\delta_\gamma$ denotes the Dirac measure at $\gamma$.
According to the general large sieve inequality (see \cite[Prop.~2.3]{k}), we have the estimate
\begin{equation}\label{large sieve}
\left|S_T\left(\Gamma,\{\Omega_{\mathfrak{p}}\}_{\mathfrak{p}\in\mathcal{L}^*};\mathcal{L}^*\right)\right|
 \le \Delta\cdot V\left(\{\Omega_{\mathfrak{p}}\}_{\mathfrak{p}\in\mathcal{L}^*}\right)^{-1},
\end{equation}
where $\Delta=\Delta(T,\mathcal{L})$ is the large sieve constant,
namely, the smallest number such that
\begin{equation}\label{spectral estimate}
{\sum_{\mathfrak{a}\in\mathcal{L}}}\sum_{\phi\in \mathcal{B}_{\mathfrak{a}}\backslash\{1\}}
\left|\int_\Gamma \alpha(\gamma)\phi(\pi_{\mathfrak{p}}(\gamma))\, d\mu_T(\gamma)\right|^2
\le \Delta \int_\Gamma |\alpha(\gamma)|^2\, d\mu_T(\gamma)
\end{equation}
for all $\alpha\in L^2(\Gamma,\mu_T)$.

For an ideal $\mathfrak{a}$ coprime with $S$,
we set 
$$
\Gamma(\mathfrak{a})=\{\gamma\in\Gamma:\, \gamma=\hbox{Id}\mod \mathfrak{a}\}.
$$
Then by \cite[Th.~5.1]{gn}, there exist $T_0,\delta>0$ such that for 
all ideals $\mathfrak{a}$ of $O_S$, $\gamma_0\in\Gamma$ and $T\ge T_0$, we have
\begin{equation}\label{eq:count}
|\{\gamma\in\gamma_0\Gamma(\mathfrak{a}):\, \hbox{H}(\gamma)\le T\}|=\frac{N_T(\Gamma)}{|\Gamma:\Gamma(\mathfrak{a})|}
+ O\left(N_T(\Gamma)^{1-\rho}\right).
\end{equation}
It follows from the strong approximation property of ${\sf G}$ that excluding a finite
set of primes $R$, we may assume that the reduction map $\pi_{\mathfrak{a}}$
is surjective for all $\mathfrak{a}\in\mathcal{L}$. In particular, this implies that
$Y_{\mathfrak{a}}\simeq \Gamma/\Gamma(\mathfrak{a})$, and we deduce that
for all $\mathfrak{a}\in \mathcal{L}$, $y\in Y_{\mathfrak{a}}$ and $T\ge T_0$,
we have 
\begin{equation}\label{eq:gn}
\mu_T\left(\{\pi_{\mathfrak{a}}(\gamma)=y\}\right)=\frac{N_T(\Gamma)}{|Y_{\mathfrak{a}}|}
+ O\left(N_T(\Gamma)^{1-\rho}\right).
\end{equation}
The implied constant depends only on $\Gamma$.

Given ideals $\mathfrak{a},\mathfrak{b}\in\mathcal{L}$, we denote by $\mathfrak{d}$
their greatest common divisor and by $[\mathfrak{a},\mathfrak{b}]$
their least common multiple. Then
$$
Y_{\mathfrak{a}}\simeq Y_{\mathfrak{a}'}\times Y_{\mathfrak{d}},\quad
Y_{\mathfrak{b}}\simeq Y_{\mathfrak{d}}\times Y_{\mathfrak{b}'},\quad
Y_{[\mathfrak{a},\mathfrak{b}]}\simeq Y_{\mathfrak{a}'}\times Y_{\mathfrak{d}}\times Y_{\mathfrak{b}'}
$$
where we write $\mathfrak{a}=\mathfrak{a}'\mathfrak{d}$ and $\mathfrak{b}=\mathfrak{d}\mathfrak{b}'$.
Then every $\phi\in\mathcal{B}_{\mathfrak{a}}$ and $\psi\in\mathcal{B}_{\mathfrak{b}}$
can be written as
$$
\phi=\phi_1\otimes \phi_0\quad\hbox{and}\quad \psi=\psi_0\otimes \psi_1
$$
for some elements $\phi_1\in \mathcal{B}_{\mathfrak{a}'}$, $\phi_0,\phi_1\in \mathcal{B}_{\mathfrak{d}}$,
$\psi_1\in \mathcal{B}_{\mathfrak{b}'}$. 
Given $\phi\in\mathcal{B}_{\mathfrak{a}}$ and $\psi\in\mathcal{B}_{\mathfrak{b}}$,
we define a function on $Y_{[\mathfrak{a},\mathfrak{b}]}$ by
$$
[\phi,\bar\psi]= \phi_1\otimes (\phi_0\bar\psi_0)\otimes \bar \psi_1.
$$
Now to estimate the large sieve constant $\Delta$, we apply \cite[Cor.~2.13]{k}.
Using \eqref{eq:gn}, we obtain for some $C>0$ and all $T\ge T_0$,
\begin{align*}
\Delta\le N_T(\Gamma)+ C\, N_T(\Gamma)^{1-\rho}
\max_{\mathfrak{a}\in\mathcal{L},\phi\in\mathcal{B}_{\mathfrak{a}}} 
\sum_{\mathfrak{b}\in\mathcal{L}}|Y_{[\mathfrak{a},\mathfrak{b}]}|
\left(\sum_{\psi\in\mathcal{B}_{\mathfrak{b}}} \|[\phi,\bar\psi]\|_\infty\right).
\end{align*}
Since $\|\phi\|_2=\|\psi\|_2=1$, we obtain
$$
\|\phi\|_\infty \le |Y_{\mathfrak{a}}|^{1/2} \left(\frac{1}{|Y_{\mathfrak{a}}|}\sum_{y\in
    Y_{\mathfrak{a}}}
|\phi(y)|^2 \right)^{1/2}=|Y_{\mathfrak{a}}|^{1/2},
$$
and similarly
$$
\|\psi\|_\infty \le |Y_{\mathfrak{b}}|^{1/2}.
$$
Hence,
\begin{align*}
\|[\phi,\bar\psi]\|_\infty&\le \|\phi_1\|_\infty\cdot \|\phi_0\bar\psi_0\|_\infty\cdot \|\psi_1\|_\infty
\le \|\phi_1\|_\infty\cdot \|\phi_0\|_\infty\cdot\|\bar\psi_0\|_\infty\cdot \|\psi_1\|_\infty
\\
&\le \|\phi\|_\infty\cdot \|\psi\|_\infty
\le |Y_{\mathfrak{a}}|^{1/2}\cdot |Y_{\mathfrak{b}}|^{1/2},
\end{align*}
and it follows that
\begin{align*}
\Delta\le N_T(\Gamma)+ C\,N_T(\Gamma)^{1-\rho}
\max_{\mathfrak{a}\in\mathcal{L}} 
\sum_{\mathfrak{b}\in\mathcal{L}}|Y_{[\mathfrak{a},\mathfrak{b}]}|
\cdot |Y_{\mathfrak{a}}|^{1/2}\cdot |Y_{\mathfrak{b}}|^{1/2},
\end{align*}
which completes the proof.
\end{proof}

In conclusion, we note that \eqref{eq:count} (see \cite[Th.~5.1]{gn}) holds for
\begin{equation}\label{eq:rho}
\rho<a(a+[F:\mathbb{Q}]\dim({\sf G}))^{-1}(2n_e(p_S))^{-1}.
\end{equation}
with notation as in \cite[Sec.~4]{gn}. Hence, Proposition \ref{p:large}
holds for this choice of $\rho$ as well.

\section{Proof of the main theorems}

For $\gamma\in \Gamma$, we denote by ${\sf D}_\gamma$ the algebraic group
generated by $\gamma$. Let $\Gamma^*\subset\Gamma$ be the subsets of $\gamma$'s
such that ${\sf D}_\gamma$ is a maximal torus in $\sf G$. In particular,
every element in $\Gamma^*$ is semisimple and regular. By \cite[Lem.~2.5]{jkz},
there exists a regular function $h$ on $\sf G$ defined over $F$ such that
for $\gamma\in \Gamma$, the condition  $h(\gamma)\ne 0$ implies that $\gamma\in \Gamma^*$.
Therefore, applying \cite[Th.~1.8]{gn2} to the variety $\{h=0\}$,
we deduce that for some $\sigma>0$,
\begin{equation}\label{eq:regular}
|\{\gamma\in \Gamma^*:\, \hbox{H}(\gamma)\le T\}|=N_T(\Gamma)+O\left(N_T(\Gamma)^{1-\sigma}\right).
\end{equation}
In fact, \cite[Th.~1.8]{gn2} gives $\sigma<\rho/\dim({\sf G})$ with $\rho$ as in \eqref{eq:rho}.
This shows that it will be sufficient to produce a favourable estimate for the set $\Gamma^*$.
We note that for $\gamma\in \Gamma^*$, the maximal torus ${\sf D}_\gamma$ is split over $F_\gamma$,
and hence $F_\gamma\supset F_{\sf G}$.

Let $\sf T$ be a maximal torus of $\sf G$ defined over $F$.
The Galois group $\hbox{Gal}(F_{\sf T}/F)$ acts faithfully on the character group $X({\sf T})$,
and we denote by 
$$
\phi_{\sf T}:\hbox{Gal}(F_{\sf T}/F)\to \hbox{Aut}(X({\sf T}))
$$
the corresponding injective homomorphism. For $\gamma\in\Gamma^*$, we also use notatation
$$
\phi_{\gamma}:\hbox{Gal}(F_{\gamma}/F)\to \hbox{Aut}(X({\sf D}_\gamma)).
$$
The Weyl group $W({\sf G},{\sf T}):=N_{\sf G}({\sf T})/C_{\sf G}({\sf T})$
also acts on $X({\sf T})$. Let $\Pi({\sf G}, {\sf T})$ be the subgroup of 
$\hbox{Aut}(X({\sf T}))$ generated by $\phi_{\sf T}(\hbox{Gal}(F_{\sf T}/F))$ and
$W({\sf G},{\sf T})$. Using that all maximal tori are conjugate, one can check (see
\cite[Prop.~2.1]{jkz}) that all groups $\Pi({\sf G}, {\sf T})$ and all groups
$W({\sf G},{\sf T})$ are isomorphic. 
Because of this, we use notation  $\Pi({\sf G})$ and $W({\sf G})$.
These isomorphisms are defined uniquely up to 
composition with inner automorphisms, so that the bijections between the conjugacy
classes are canonically defined.

By \cite[Lem.~2.2]{jkz}, for $\gamma\in\Gamma^*$, we have
$$
\phi_{\gamma}(\hbox{Gal}(F_{\gamma}/F_{\sf G}))\subset W({\sf G}).
$$
Therefore, $\hbox{Gal}(F_{\gamma}/F_{\sf G})$ is isomorphic to a subgroup of $W({\sf G})$,
and to prove Theorem \ref{th:split}, it remains to show that `typically' the map
$\phi_\gamma$ is onto. 

Let $E$ be a finite extension of $F$ such that ${\sf G}$ is split over $E$.
We shall show that `typically' $\phi_{\gamma}(\hbox{Gal}(E_{\gamma}/E))$
intersects every conjugacy class of $W({\sf G})$. Then 
$\phi_{\gamma}(\hbox{Gal}(E_{\gamma}/E))^{W({\sf G})}=W({\sf G})$, and comparing cardinalities
we conclude that, in fact, $\phi_{\gamma}(\hbox{Gal}(E_{\gamma}/E))=W({\sf G})$.
Since $E\supset F_{\sf G}$, this implies that 
$\phi_{\gamma}(\hbox{Gal}(F_{\gamma}/F_{\sf G}))=W({\sf G})$ as well.

For a prime ideal $\mathfrak{p}$ of the ring of integers of $E$, we denote
by ${\sf G}^{(\mathfrak{p})}$ the reduction of $\sf G$ modulo $\mathfrak{p}$.
For all but finitely many $\mathfrak{p}$, the group ${\sf G}^{(\mathfrak{p})}$
is geometrically irreducible and split. 
Let $Y_{\mathfrak{p}}={\sf G}^{(\mathfrak{p})}({\bf F}_{\mathfrak{p}})$ and 
$Y_{\mathfrak{p}}^*$ be the subset of regular semisimple elements in $Y_{\mathfrak{p}}$.
Every $g\in Y_{\mathfrak{p}}^*$ is contained in unique maximal torus of ${\sf G}^{(\mathfrak{p})}$,
which we denote by ${\sf D}_g$. As above, for $g\in Y_{\mathfrak{p}}^*$, we have 
a homomorphism
$$
\phi_g:\hbox{Gal}(\overline{\bf F}_{\mathfrak{p}}/{\bf F}_{\mathfrak{p}})\to \hbox{Aut}(X({\sf D}_g)),
$$
and 
$$
\phi_g\left(\hbox{Gal}(\overline{\bf F}_{\mathfrak{p}}/{\bf F}_{\mathfrak{p}})\right)\subset W\left({\sf
    G}^{(\mathfrak{p})}\right)
$$
for all but finitely many $\mathfrak{p}$.
We denote by $\hbox{\rm Frob}_{\mathfrak{p}}$ the conjugacy class
in $\hbox{Gal}(\overline{\bf F}_{\mathfrak{p}}/{\bf F}_{\mathfrak{p}})$
generated by the Frobenius automorphism 
$x\mapsto x^{N\mathfrak{p}}$. For a prime ideal $\mathfrak{p}$ of $E$
which is unramified over $E_\gamma$ we denote by $\hbox{Frob}^{E_\gamma/E}_{\mathfrak{p}}$ the 
Frobenius conjugacy class in $\hbox{Gal}(E_\gamma/E)$.
By \cite[Prop.~3.1]{jkz}, 
there exists a finite set $R$ of prime ideals $\mathfrak{p}$
and a nonzero regular function $h$ on ${\sf G}$ defined over $E$ such that 
for all $\mathfrak{p}\notin R$ and every $\gamma\in \Gamma^*$
satisfying $h(\gamma)\ne 0\mod \mathfrak{p}$, we have
\begin{itemize}
\item $\pi_{\mathfrak{p}}(\gamma)\in Y_{\mathfrak{p}}^*$,
\item $\mathfrak{p}$ is unramified in $F_\gamma$,
\item $W({\sf G})\simeq W({\sf G}^{(\mathfrak{p})})$, and 
there is a canonical bijection between the sets of conjugacy classes
in $W({\sf G})$ and $W({\sf G}^{(\mathfrak{p})})$ such that 
the conjugacy classes $\phi_\gamma\left(\hbox{\rm Frob}^{E_\gamma/E}_{\mathfrak{p}}\right)$
and  $\phi_{\pi_{\mathfrak{p}}(\gamma)}\left(\hbox{\rm Frob}_{\mathfrak{p}}\right)$
correspond to each other.
\end{itemize}
Moreover, enlarging $R$ if necessary, we may assume that 
for $\mathfrak{p}\notin R$, $h\ne 0\mod\mathfrak{p}$,
and for prime ideals $\mathfrak{q}$ of $F$ dividing $\mathfrak{p}$,
${\sf G}^{(\mathfrak{p})}\simeq {\sf G}^{(\mathfrak{q})}$ and Proposition \ref{p:large} applies.

Now we apply Proposition \ref{p:large} with $\mathcal{L}=\mathcal{L}^*$ being the set of prime
ideals $\mathfrak{p}$
which are not in $R$,
split completely in the extension $E/F$,
 and satisfy $N\mathfrak{p}\le L$ for a parameter $L\ge 2$ that will be chosen later. 
We note that by the Chebotarev density theorem, $|\mathcal{L}|\gg \frac{L}{\log L}$.
The assumption that a prime ideal $\mathfrak{p}\in\mathcal{L}$ splits completely guarantees that
for every prime ideal $\mathfrak{q}$ of $F$ that divides $\mathfrak{p}$, 
we have $Y_{\mathfrak{p}}\simeq Y_{\mathfrak{q}}$, and hence Proposition \ref{p:large}
applies to the maps $\pi_{\mathfrak{p}}:\Gamma\to Y_{\mathfrak{p}}$.
We fix a conjugacy class $\mathcal{C}\subset W({\sf G})\simeq W({\sf G}^{(\mathfrak{p})})$
and for $\mathfrak{p}\in \mathcal{L}$ consider a set
$$
\Omega_{\mathfrak{p}}^{\mathcal{C}}=Y_{\mathfrak{p}}
\backslash 
\left\{g\in Y_{\mathfrak{p}}^*:\,\, h(g)\ne  0,\, 
\phi_{g}(\hbox{\rm Frob}_{\mathfrak{p}})=\mathcal{C}\right\}.
$$
In order to estimate 
$S_T\left(\Gamma,\left\{\Omega_p^{\mathcal{C}}\right\}_{\mathfrak{p}\in\mathcal{L}};\mathcal{L}\right)$,
we need to establish a lower bound for
$V\left(\left\{\Omega_p^{\mathcal{C}}\right\}_{\mathfrak{p}\in\mathcal{L}}\right)$
and an upper bound for $M(\mathcal{L})$.

Since $\{h=0\}$ is a subvariety of ${\sf G}^{(\mathfrak{p})}$ with smaller dimension, it follows that
$$
\left|\left\{g\in {\sf G}^{(\mathfrak{p})}({\bf F}_{\mathfrak{p}}):\,\, h(g)=  0\right\}\right|
\ll |{\bf F}_{\mathfrak{p}}|^{\dim({\sf G})-1}\ll 
|Y_{\mathfrak{p}}|/|{\bf F}_{\mathfrak{p}}|.
$$
Also, by \cite[Prop.~4.1]{jkz},
$$
|\{ g\in Y_{\mathfrak{p}}^*:\,
\phi_g(\hbox{\rm Frob}_{\mathfrak{p}})=\mathcal{C}\}|
=\frac{|C|}{|W({\sf G})|} |Y_{\mathfrak{p}}|\left(1+O\left(|{\bf
  F}_{\mathfrak{p}}|^{-1}\right)\right).
$$
Therefore,
$$
\frac{|\Omega_{\mathfrak{p}}^{\mathcal{C}}|}{\left|Y_{\mathfrak{p}}\right|}
=\frac{|W({\sf G})-\mathcal{C}|}{|W({\sf G})|} +O\left(|{\bf
  F}_{\mathfrak{p}}|^{-1}\right),
$$
and it follows that
\begin{equation}\label{eq:v_est}
V\left(\left\{\Omega_p^{\mathcal{C}}\right\}_{\mathfrak{p}\in\mathcal{L}}\right)\ge \sum_{\mathfrak{p}\in\mathcal{L}} \frac{|\Omega_{\mathfrak{p}}^{\mathcal{C}}|}
{|Y_{\mathfrak{p}}|}\gg |\mathcal{L}|\gg
\frac{L}{\log L}.
\end{equation}

Next we estimate $M(\mathcal{L})$. For $\mathfrak{p},\mathfrak{q}\in\mathcal{L}$, we have
\begin{align*}
\left|Y_{\mathfrak{p}}\right|&\ll |{\bf F}_{\mathfrak{p}}|^{\dim({\sf G})}\le L^{\dim({\sf G})},\\
\left|Y_{[\mathfrak{p},\mathfrak{q}]}\right|&\le \left|Y_{\mathfrak{p}}\right|\cdot \left|Y_{\mathfrak{q}}\right|
\ll L^{2\dim({\sf G})},
\end{align*}
and it follows that
\begin{equation}\label{eq:m_est}
M(\mathcal{L})\le L^{3\dim({\sf G})} |\mathcal{L}|
\ll \frac{L^{3\dim({\sf G})+1}}{\log L}.
\end{equation}
Now Proposition \ref{p:large}, together with \eqref{eq:v_est} and \eqref{eq:m_est}, implies that
$$
S_T\left(\Gamma,\left\{\Omega_p^{\mathcal{C}}\right\}_{\mathfrak{p}\in\mathcal{L}};\mathcal{L}\right)\ll \left(N_T(\Gamma)+N_T(\Gamma)^{1-\rho} \frac{L^{3\dim({\sf G})+1}}{\log L}\right)
\frac{\log L}{L}.
$$
Taking $L=N_T(\Gamma)^{\rho/(3\dim({\sf G})+1)}$, we deduce that for $\delta<\rho/(3\dim({\sf G})+1)$,
we have
$$
S_T\left(\Gamma,\left\{\Omega_p^{\mathcal{C}}\right\}_{\mathfrak{p}\in\mathcal{L}};\mathcal{L}\right)
\ll  N_T(\Gamma)^{1-\delta}.
$$
Combining this estimate with \eqref{eq:regular}, we deduce that
$$
\left|\left\{\gamma\in\Gamma^*:\,\, 
\begin{tabular}{c}
$\hbox{\rm H}(\gamma)\le T$\\
$\exists\,\mathfrak{p}:\, \phi_\gamma\left(\hbox{\rm Frob}^{E_\gamma/E}_{\mathfrak{p}}\right)=\mathcal{C}$
\end{tabular}
\right\}\right|
=N_T(\Gamma)+O(N_T(\Gamma)^{1-\delta}),
$$
since $\delta<\sigma$.
This estimate holds for all cosets $\mathcal{C}$ of the Weyl group $W({\sf G})$.
Therefore,
$$
\left|\left\{\gamma\in\Gamma^*:\,\, 
\begin{tabular}{c}
$\hbox{\rm H}(\gamma)\le T$\\
$\forall\,\mathcal{C}:\, \phi_\gamma(\hbox{\rm Gal}(E_\gamma/E))\cap\mathcal{C}\ne \emptyset$
\end{tabular}
\right\}\right|
=N_T(\Gamma)+O(N_T(\Gamma)^{1-\delta}).
$$
As it was remarked above, this implies that
$$
\left|\left\{\gamma\in\Gamma^*:\,\, 
\begin{tabular}{c}
$\hbox{\rm H}(\gamma)\le T$\\
$\hbox{\rm Gal}(E_\gamma/E)\simeq W({\sf G})$
\end{tabular}
\right\}\right|
=N_T(\Gamma)+O(N_T(\Gamma)^{1-\delta}),
$$
and
$$
\left|\left\{\gamma\in\Gamma^*:\,\, 
\begin{tabular}{c}
$\hbox{\rm H}(\gamma)\le T$\\
$\hbox{\rm Gal}(F_\gamma/F_{\sf G})\simeq W({\sf G})$
\end{tabular}
\right\}\right|
=N_T(\Gamma)+O(N_T(\Gamma)^{1-\delta}).
$$
Now Theorem \ref{th:split} follows from \eqref{eq:regular}.

To prove Theorem \ref{th:general}, we observe that if for $\gamma\in\Gamma^*$,
we have 
$$
\phi_\gamma(\hbox{\rm Gal}(F_\gamma/F_{\sf G}))=W({\sf G}),
$$
then
$$
\phi_\gamma(\hbox{\rm Gal}(F_\gamma/F))=\Pi({\sf G}),
$$
and $\phi_\gamma$ defines an isomorphism $\hbox{\rm Gal}(F_\gamma/F)\simeq\Pi({\sf G})$.
Therefore, it follows from the above argument that
$$
\left|\left\{\gamma\in\Gamma^*:\,\, 
\begin{tabular}{c}
$\hbox{\rm H}(\gamma)\le T$\\
$\hbox{\rm Gal}(F_\gamma/F)\simeq \Pi({\sf G})$
\end{tabular}
\right\}\right|
=N_T(\Gamma)+O(N_T(\Gamma)^{1-\delta}),
$$
and finally Theorem \ref{th:general} follows from \eqref{eq:regular}.
The estimate \eqref{eq:delta} on $\delta$ follows from \eqref{eq:rho}.

%
%
%
%
%

\end{document}